\documentclass[11pt,reqno]{amsart}

\usepackage{fullpage}
\usepackage{amsmath}
\usepackage{amscd}
\usepackage{verbatim}
\usepackage{amsfonts}
\usepackage{amsthm}
\usepackage{amssymb}
\usepackage{graphicx}
\usepackage{epsfig}
\usepackage[cmtip,all]{xy}
\usepackage{url}
\usepackage{booktabs}

\newtheorem{theorem}{Theorem}[section]
\newtheorem{lemma}[theorem]{Lemma}

\newtheorem{proposition}[theorem]{Proposition}
\newtheorem{conjecture}[theorem]{Conjecture}

\theoremstyle{definition}
\newtheorem{definition}[theorem]{Definition}
\newtheorem{example}[theorem]{Example}

\theoremstyle{remark}

\newtheorem*{remark*}{Remark}

\newcommand{\ra}[1]{\renewcommand{\arraystretch}{#1}}

\numberwithin{equation}{section}

\begin{document}

\title{Set-theoretic Yang-Baxter (co)homology theory of involutive non-degenerate solutions}

\author{J\'ozef H. Przytycki}

\author{Petr Vojt\v{e}chovsk\'y}

\author{Seung Yeop Yang}

\email[J\'ozef H. Przytycki]{przytyck@gwu.edu}

\email[Petr Vojt\v{e}chovsk\'y]{petr@math.du.edu}

\email[Seung Yeop Yang]{seungyeop.yang@knu.ac.kr}

\address{Department of Mathematics, The George Washington University, Washington, DC 20052, USA and University of Gda\'nsk, Poland}

\address{Department of Mathematics, University of Denver, 2390 S York St, Denver, Colorado, 80208, USA}

\address{Department of Mathematics, Kyungpook National University, Daegu, 41566, Republic of Korea}

\begin{abstract}
W. Rump showed that there exists a one-to-one correspondence between involutive right non-degenerate solutions of the Yang-Baxter equation and cycle sets. J. S. Carter, M. Elhamdadi, and M. Saito, meanwhile, introduced a homology theory of set-theoretic solutions of the Yang-Baxter equation in order to define cocycle invariants of classical knots. In this paper, we introduce the normalized homology theory of an involutive right non-degenerate solution of the Yang-Baxter equation and compute the normalized set-theoretic Yang-Baxter homology of cyclic racks. Moreover, we explicitly calculate some $2$-cocycles, which can be used to classify certain families of torus links.
\end{abstract}

\keywords{set-theoretical solution of Yang-Baxter equation, cycle set, normalized Yang-Baxter (co)homology theory}

\subjclass[2010]{16T25, 20N05, 57M27}

\maketitle

\section{Introduction}

For a given set $X,$ a \emph{set-theoretic solution} of the Yang-Baxter equation is a function $R:X\times X\to X\times X$ satisfying the following equation
$$(R \times \textrm{Id}_{X})(\textrm{Id}_{X} \times R)(R \times \textrm{Id}_{X}) = (\textrm{Id}_{X} \times R)(R \times \textrm{Id}_{X})(\textrm{Id}_{X} \times R),$$
where $\textrm{Id}_{X}$ is the identity map on $X.$ A solution $R$ is said to be \emph{left non-degenerate} if $R_{1}(x, - )$ is bijective for every $x \in X$ and \emph{right non-degenerate} if $R_{2}(- , y)$ is bijective for every $y \in X,$ where $R(x,y)=(R_{1}(x,y), R_{2}(x,y)).$ We call $R$ \emph{non-degenerate} if both $R_{1}(x, - )$ and $R_{2}(- , y)$ are bijective. A solution $R$ is said to be \emph{involutive} if $R^{2} = \textrm{Id}_{X \times X}.$

A \emph{magma} $(X,\cdot)$ is a set $X$ equipped with a binary operation $\cdot: X \rightarrow X.$ A magma $(X,\cdot)$ is called a \emph{left quasigroup} if all its left translations $l_{x}:X\to X$, $y\mapsto xy$ are permutations, and a \emph{right quasigroup} if all its right translations $r_{x}:X \to X$, $y \mapsto yx$ are permutations. A magma that is both left quasigroup and right quasigroup will be called \emph{latin} or \emph{quasigroup}. In a right quasigroup, we denote $r_y^{-1}(x)$ by $x/y.$ A right quasigroup is called a \emph{cycle set} or a \emph{Rump right quasigroup} if it satisfies the following identity
\begin{equation}\label{Eq:Rump}
    (zx)(yx) = (zy)(xy).
\end{equation}

The above identity was first studied by Bosbach \cite{Bos} and Traczyk \cite{Tra}. Rump \cite{Rum} showed the importance of this algebraic structure by constructing a relationship between it and an involutive non-degenerate set-theoretic solution of the Yang-Baxter equation.

\begin{theorem}\cite{Rum}\label{thm: Rump}
For a given set $X,$ we let $R:X\times X\to X\times X$ be a map.
There is a one-to-one correspondence between involutive right non-degenerate solutions of the Yang-Baxter equation on $X$ and cycle sets on $X,$ namely
  $$R \mapsto (X,\cdot),~ xy = R_{2}(- , y)^{-1}(x); \hspace{1cm} (X,\cdot) \mapsto R,~ R(x,y) = (y(x / y),x / y).$$
\end{theorem}

In the theory of general Yang-Baxter operators, the operators satisfying quadratic equation $R^{2} = x R + y \textrm{Id}_{X}$ are of special value because, as shown by Jones, they lead to Jones and Homflypt polynomial invariants of knots and links (see \cite{Jon, Tur}). For example, column unital Yang-Baxter operators (e.g., stochastic matrices) were discussed in \cite{PW}. In the set-theoretic world two special cases remain: $R^{2} = \textrm{Id}_{X}$ and $R^{2} = R.$ The involutory case is discussed in this paper and the idempotent case is discussed in \cite{SV}.

Carter, Elhamdadi, and Saito \cite{CES} introduced a homology theory of set-theoretic solutions of the Yang-Baxter equation, and defined cocycle knot invariants in a state-sum formulation. This theory is generalized and modified to obtain invariants of virtual links \cite{CN} and handlebody-links \cite{IIKKMO}, etc.

We investigate basic notions and properties of cycle sets in Section \ref{Section 1.1}. In Section \ref{Section 2}, we introduce the normalized homology theory of an involutive non-degenerate solution of the Yang-Baxter equation. We also show that set-theoretic Yang-Baxter homology groups of a certain family of solutions split into normalized and degenerate parts. In Section \ref{Section 3}, we explain how the cocycles of normalized set-theoretic Yang-Baxter cohomology groups can be used to define invariants of links of codimension two.

\subsection{Preliminaries}\label{Section 1.1}

A magma $(X,\cdot)$ is said to be \emph{uniquely $2$-divisible} if the squaring function $x \mapsto x^{2}$ is bijective and is said to be \emph{$\Delta$-bijective} if the function $\Delta: X \times X \rightarrow X \times X$ defined by $\Delta(x,y)=(xy,yx)$ is bijective. It is easy to check that every bijective magma is uniquely $2$-divisible. Rump remarked on basic properties of cycle sets.

\begin{proposition}\cite{Rum}\label{Prop:Rump}
\begin{enumerate}
  \item Every uniquely $2$-divisible cycle set is $\Delta$-bijective.
  \item Every finite cycle set is uniquely $2$-divisible.\footnote{See \cite{BKSV} for a short proof.}
  \item A cycle set is uniquely $2$-divisible if and only if the corresponding solution is non-degenerate.
\end{enumerate}
\end{proposition}

Racks and quandles, which are well-known set-theoretic solutions of the Yang-Baxter equation, are non-associative algebraic structures satisfying axioms motivated by the Reidemeister moves. The precise definitions are as follows. A \emph{rack} is a magma $(X,*)$ such that:
\begin{enumerate}
  \item (Right self-distributivity) For every $a,b,c \in X,$ $(a*b)*c = (a*c)*(b*c).$
  \item (Invertibility) For every $b \in X,$ the mapping $r_{b}$ is bijective.
\end{enumerate}
If a rack $(X,*)$ satisfies the idempotency, i.e., for every $a \in X$ we have $a*a=a,$ then we call it a \emph{quandle}. We can obtain basic examples of racks and quandles using cyclic groups. We denote the cyclic group of order $n$ by $\mathbb{Z}_{n}.$
\begin{enumerate}
  \item We call $\mathbb{Z}_{n}$ with the operation $i*j=i+1$ (mod $n$) the \emph{cyclic rack} of order $n$ and denote it by $C_{n}.$
  \item $\mathbb{Z}_{n}$ with the operation $i*j=2j-i$ (mod $n$) is called the \emph{dihedral quandle} of order $n,$ denoted by $R_{n}.$
  \item A module $M$ over the Laurent polynomial ring $\mathbb{Z}[t^{\pm1}]$ with the operation $a * b = ta + (1-t)b$ is called an \emph{Alexander quandle}.
\end{enumerate}
It is easy to check that every cyclic rack is a cycle set. In order for an Alexander quandle $M$ to become a cycle set, all elements of $M$ have to be annihilated by $(1-t)^{2}$:

\begin{proposition}
Let $M$ be an Alexander quandle. Then $M$ is a cycle set if and only if $(1-t)^{2}$ annihilates all elements of $M.$
\end{proposition}

\begin{proof}
Note that for all $x,y,z\in M:$
\begin{align*}
& ~\quad (z*x)*(y*x) = (z*y)*(x*y)\\
& \Leftrightarrow (tz+(1-t)x)*(ty+(1-t)x) = (tz+(1-t)y)*(tx+(1-t)y)\\
& \Leftrightarrow (1-t)^{2}x = (1-t)^{2}y\\
& \Leftrightarrow (1-t)^{2}(x-y) = 0.
\end{align*}
\end{proof}

Similarly, one can check that the dihedral quandle $R_{n}$ is a cycle set if and only if $n$ divides $4,$ i.e., only $R_{4}$ is a non-trivial\footnote{A quandle $X$ is said to be \emph{trivial} if $x*y=x$ for all $x, y \in X.$} dihedral quandle which is also a cycle set.\\

Let $(G,+)$ be an abelian group. For given endomorphisms $\phi, \psi$ of $(G,+)$ and $c \in G,$ we define the binary operation $x*y = \phi(x)+\psi(y)+c.$ The magma $(G,*)$ is said to be \emph{affine over} $(G,+)$ and is denoted by $\textrm{Aff}(G, \phi, \psi, c).$ Note that $\textrm{Aff}(G, \phi, \psi, c)$ forms a right quasigroup (resp., quasigroup) if and only if $\phi \in \textrm{Aut}(G,+)$ (resp., $\phi, \psi \in \textrm{Aut}(G,+)$).

Denote by $\mathbb F_q$ the finite field of order $q$ and note that the automorphism group of $(\mathbb F_q^n,+)$ is the general linear group $GL_n(q)$. A large class of affine cycle sets was obtained in \cite{BKSV} from invertible matrices $A$, $B$ satisfying $[A,B]=A^2$. Table \ref{table 1} gives a multiplication table of the affine cycle set $\mathrm{Aff}(\mathbb F_2^2,\binom{1\ 0}{1\ 1},\binom{0\ 1}{1\ 0},\binom{0}{0})$, and Table \ref{table 2} gives a multiplication table of the affine cycle set $\mathrm{Aff}(\mathbb F_4^2,\binom{0\ u}{u^2\ 0},\binom{0\ 1}{1\ 0},\binom{0}{0})$, where $u$ is a primitive element of $\mathbb F_4.$ Neither of the two cycle sets is a rack.

\begin{table}[h]
  \centering
  \caption{An affine cycle set of order $4$}\label{table 1}
  \begin{tabular}{c|cccc}
 $\ast$ & 1 & 2 & 3 & 4 \\
 \hline
 1 & 1 & 3 & 2 & 4  \\
 2 & 2 & 4 & 1 & 3  \\
 3 & 4 & 2 & 3 & 1  \\
 4 & 3 & 1 & 4 & 2

\end{tabular}
\end{table}

\section{Normalized set-theoretic Yang-Baxter (co)homology of cycle sets}\label{Section 2}

Given a set $X,$ we let $R:X \times X \rightarrow X \times X$ be a solution of the set-theoretic Yang-Baxter equation on $X.$ Consider two maps $\mu,\nu:X \times X \rightarrow X$ defined by $\mu(x,y)=R_{2}(x,y)$ and $\nu(x,y)=R_{1}(x,y),$ where $R(x,y)=(R_{1}(x,y),R_{2}(x,y)).$
Let $C_{n}^{YB}(X)$ be the free abelian group generated by $n$-tuples $(x_{1},\ldots ,x_{n})$ of elements of $X.$ Let us define an $n$-dimensional boundary homomorphism $\partial_{n}^{YB}: C_{n}^{YB}(X) \rightarrow C_{n-1}^{YB}(X)$ by $\sum\limits_{i=1}^{n}(-1)^{i+1}(d_{i,n}^{l}-d_{i,n}^{r}),$ where the two face maps $d_{i,n}^{l},d_{i,n}^{r}:C_{n}^{YB}(X) \rightarrow C_{n-1}^{YB}(X)$ are given by
$$d_{i,n}^{l} = (\mu \times \textrm{Id}_{X}^{\times (n-2)}) \circ (\textrm{Id}_{X} \times R \times \textrm{Id}_{X}^{\times (n-3)}) \circ \cdots \circ (\textrm{Id}_{X}^{\times (i-2)} \times R \times \textrm{Id}_{X}^{\times (n-i)}),$$
$$d_{i,n}^{r} = (\textrm{Id}_{X}^{\times (n-2)} \times \nu) \circ (\textrm{Id}_{X}^{\times (n-3)} \times R \times \textrm{Id}_{X}) \circ \cdots \circ (\textrm{Id}_{X}^{\times (i-1)} \times R \times \textrm{Id}_{X}^{\times (n-i-1)}).$$
Since $\partial_{n-1}^{YB} \circ \partial_{n}^{YB}=0,$ $C_{*}^{YB}(X):=(C_{n}^{YB}(X),\partial_{n}^{YB})$ forms a chain complex. Note that the face maps defined above can be illustrated as in Figure \ref{YBfacemap}. See \cite{Leb, Prz} for further details.

\begin{figure}[h]
\centerline{{\psfig{figure=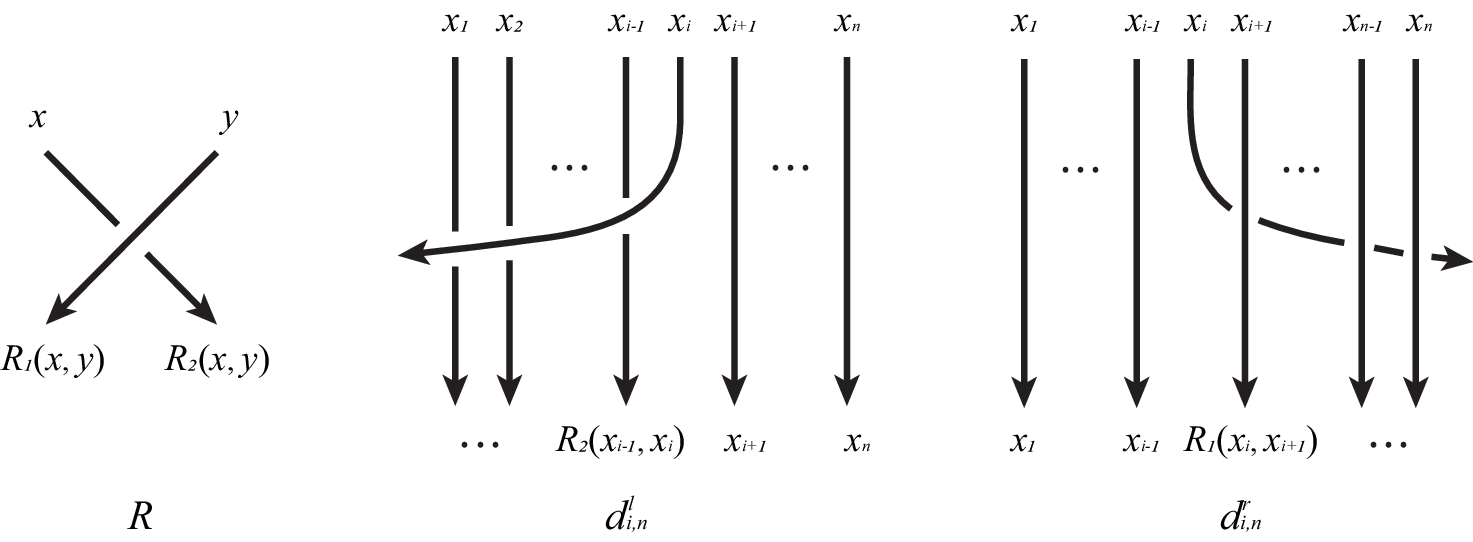,height=5.2cm}}}
\caption{Diagrammatic interpretation of Yang-Baxter face maps}
\label{YBfacemap}
\end{figure}

For a given abelian group $A,$ the yielded homology and cohomology groups $H_{*}^{YB}(X;A)$ and $H^{*}_{YB}(X;A)$ are called the \emph{set-theoretic Yang-Baxter homology and cohomology groups of $X$ with coefficients in $A.$} See \cite{CES} for details.\\

For a cycle set $X,$ we let $R:X \times X \rightarrow X \times X$ be the map given by $R(x,y)=(y(x / y),x / y).$ Then $R$ is a set-theoretic solution of the Yang-Baxter equation by Theorem \ref{thm: Rump}. Consider the subgroup $C_{n}^{D}(X)$ of $C_{n}^{YB}(X)$ defined by
$$C_{n}^{D}(X)=\text{span}\{s_{1,n}^{YB}(C_{n-1}^{YB}(X)), s_{2,n}^{YB}(C_{n-1}^{YB}(X)), \ldots , s_{n-1,n}^{YB}(C_{n-1}^{YB}(X))\},$$
where $s_{i,n}^{YB}:C_{n-1}^{YB}(X) \rightarrow C_{n}^{YB}(X)$ are the degeneracy maps given by
$$s_{i,n}^{YB}(x_{1}, \ldots ,x_{n-1})=(x_{1}, \ldots, x_{i-1}, x_{i}x_{i}, x_{i}, x_{i+1}, \ldots ,x_{n-1})$$
if $n \geq 2,$ otherwise we let $C_{n}^{D}(X)=0.$

Before we prove that the degenerate chain groups defined above form a sub-chain complex of the set-theoretic Yang-Baxter chain complex, let us consider the following identities needed in the proof.

\begin{lemma}\label{lemma: identities}
For a cycle set $X,$ the following identities hold for all $x,y \in X$:
\begin{enumerate}
  \item $(xx)(y/(xx)) = (x((y/(xx))/x))(x((y/(xx))/x)),$
  \item $(xx)(y(x/y)) = (x/y)(x/y).$
\end{enumerate}
\end{lemma}

\begin{proof}
(1) Let $x, y \in X.$ Then
\begin{flalign*}
(x((y/(xx))/x))(x((y/(xx))/x))
&=(xx)(((y/(xx))/x)x) \textrm{~by the right Rump identity}\\
&=(xx)(y/(xx)).
\end{flalign*}

(2) For every $x,y \in X,$ we have
\begin{flalign*}
(x/y)(x/y)
&=(((x/y)(x/y))(y(x/y)))/(y(x/y))\\
&=(((x/y)y)((x/y)y))/(y(x/y)) \textrm{~by the right Rump identity}\\
&=(xx)(y(x/y)).
\end{flalign*}
\end{proof}

\begin{theorem}\label{theorem: subchain}
Let $X$ be a cycle set. Then $(C_{n}^{D}(X),\partial_{n}^{YB})$ forms a sub-chain complex of $(C_{n}^{YB}(X),\partial_{n}^{YB}).$
\end{theorem}

\begin{proof}
Suppose that $X$ is a cycle set. Let $\textbf{x}=(x_{1}, \ldots, x_{i-1}, x_{i}x_{i}, x_{i}, x_{i+2}, \ldots ,x_{n}) \in C_{n}^{D}(X).$ It suffices to show that $\partial_{n}^{YB}(\textbf{x}) = \sum\limits_{j=1}^{n}(-1)^{j+1}(d_{j,n}^{l}-d_{j,n}^{r})(\textbf{x}) \in C_{n-1}^{D}(X).$\\
To show this we demonstrate that for $\varepsilon = l$ or $r,$ we have:
$$\left\{
  \begin{array}{ll}
    d_{j,n}^{\varepsilon}(\textbf{x}) \in s_{i,n}^{YB}(C_{n-1}^{YB}(X)) & \hbox{if $j > i$;} \\
    d_{j,n}^{\varepsilon}(\textbf{x}) \in s_{i-1,n}^{YB}(C_{n-1}^{YB}(X)) & \hbox{if $j < i$;} \\
    d_{i,n}^{\varepsilon}(\textbf{x})=d_{i+1,n}^{\varepsilon}(\textbf{x}).
  \end{array}
\right.$$
Denote by $y$ the $(i-1)$-th coordinate of $(\textrm{Id}_{X}^{\times (i-3)} \times R \times \textrm{Id}_{X}^{\times (n-i+1)}) \circ \cdots \circ (\textrm{Id}_{X}^{\times (j-1)} \times R \times \textrm{Id}_{X}^{\times (n-j-1)})(\textbf{x})$ and by $z$ the $(i+2)$-th coordinate of $(\textrm{Id}_{X}^{\times (i+1)} \times R \times \textrm{Id}_{X}^{\times (n-i-3)}) \circ \cdots \circ (\textrm{Id}_{X}^{\times (j-2)} \times R \times \textrm{Id}_{X}^{\times (n-j)})(\textbf{x}).$
Suppose that $j < i.$ Then $d_{j,n}^{l}(\textbf{x})=(\ldots, x_{i}x_{i}, x_{i}, \ldots) \in C_{n-1}^{D}(X)$ (see Figure \ref{subchain} $(i)$). Moreover, by Lemma \ref{lemma: identities} (1), we can check that
$$d_{j,n}^{r}(\textbf{x})=(\ldots, (x_{i}x_{i})(y/(x_{i}x_{i})), (x_{i}((y/(x_{i}x_{i}))/x_{i})), \ldots)$$
\hspace{4.9cm}$=(\ldots, (x_{i}((y/(x_{i}x_{i}))/x_{i}))(x_{i}((y/(x_{i}x_{i}))/x_{i})), (x_{i}((y/(x_{i}x_{i}))/x_{i})), \ldots)$\\
also belongs to $C_{n-1}^{D}(X)$ (see Figure \ref{subchain} $(v)$).\\
Similarly, if we assume that $j > i+1,$ then by Lemma \ref{lemma: identities} (2),
$$d_{j,n}^{l}(\textbf{x})=(\ldots, (x_{i}x_{i})/(z(x_{i}/z)), x_{i}/z, \ldots)=(\ldots, (x_{i}/z)(x_{i}/z), x_{i}/z, \ldots) \in C_{n-1}^{D}(X)$$
(see Figure \ref{subchain} $(iv)$) and, moreover, $d_{j,n}^{r}(\textbf{x})=(\ldots, x_{i}x_{i}, x_{i}, \ldots) \in C_{n-1}^{D}(X)$ (see Figure \ref{subchain} $(viii)$).\\
We now only need to consider the cases in which $j$ is equal to $i$ or $i+1.$ We see that the images of $\textbf{x}$ under $d_{i,n}^{l}$ and $d_{i+1,n}^{l}$ are the same (i.e., $d_{i,n}^{l}(\textbf{x})=(\ldots, x_{i-1}/(x_{i}x_{i}), x_{i}, \ldots)=d_{i+1,n}^{l}(\textbf{x}),$ see Figure \ref{subchain} $(ii)$ and $(iii)$). Similarly, $d_{i,n}^{r}(\textbf{x})=(\ldots, x_{i}x_{i}, x_{i+2}(x_{i}/x_{i+2}), \ldots)=d_{i+1,n}^{r}(\textbf{x}),$ see Figure \ref{subchain} $(vi)$ and $(vii).$ That means $(-1)^{i+1}(d_{i,n}^{l}-d_{i,n}^{r})(\textbf{x})+(-1)^{i+2}(d_{i+1,n}^{l}-d_{i+1,n}^{r})(\textbf{x})=0,$ therefore, $\partial_{n}^{YB}(\textbf{x}) = \sum\limits_{j=1}^{n}(-1)^{j+1}(d_{j,n}^{l}-d_{j,n}^{r})(\textbf{x}) \in C_{n-1}^{D}(X)$ as desired.
\end{proof}

\begin{remark*}
 General degeneracies in homology of non-degenerate set-theoretic solutions of the Yang-Baxter equation and criteria to obtain simplicial modules were discussed in \cite{LV}. Our Theorem \ref{theorem: subchain} is not covered by [Theorem 2.2, \cite{LV}] since, calculating in the cycle set of Table \ref{table 1}, it is not difficult to show that both (2.2) and (2.3) of \cite{LV} fail.
\end{remark*}

\begin{figure}[h]
\centerline{{\psfig{figure=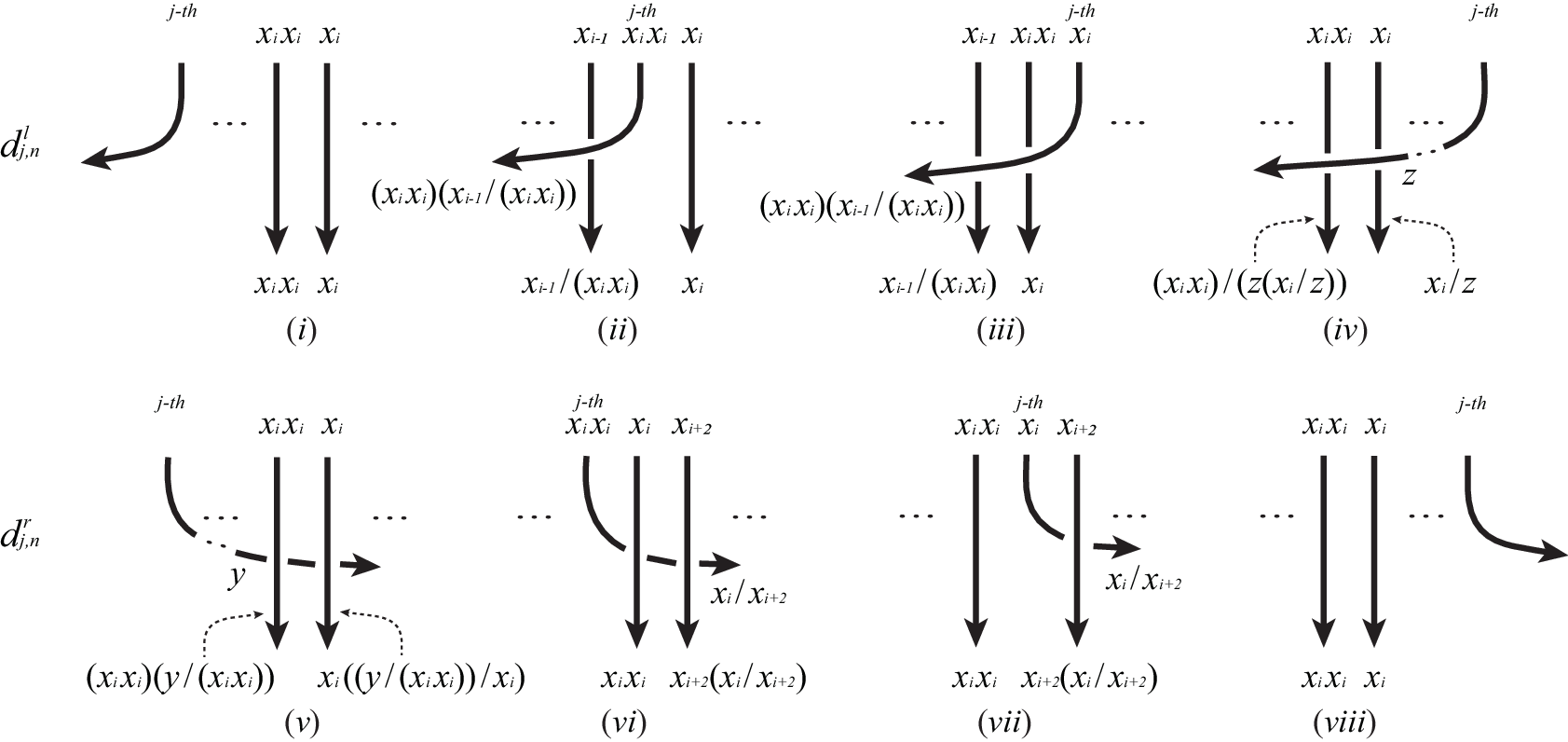,height=7.3cm}}}
\caption{$\partial_{n}^{YB}(C_{n}^{D}(X)) \subset C_{n-1}^{D}(X)$}
\label{subchain}
\end{figure}

We call $H_{n}^{D}(X;A)=H_{n}(C_{*}^{D}(X;A))$ and $H^{n}_{D}(X;A)=H^{n}(C^{*}_{D}(X;A))$ the \emph{degenerate set-theoretic Yang-Baxter homology and cohomology groups of $X$ with coefficients in $A.$}\\
We now let $C_{n}^{NYB}(X)=C_{n}^{YB}(X)/C_{n}^{D}(X),$ and define the quotient chain complex $C_{*}^{NYB}(X):=(C_{n}^{NYB}(X),\partial_{n}^{NYB}),$ where $\partial_{n}^{NYB}$ is the induced homomorphism.\\
For an abelian group $A,$ we define the chain and cochain complexes
$$C_{*}^{NYB}(X;A):= C_{*}^{NYB}(X) \otimes A, ~ \partial^{NYB} := \partial^{NYB} \otimes \textrm{Id}_{A};$$
$$C^{*}_{NYB}(X;A):= \text{Hom}(C_{*}^{NYB}(X), A), ~ \delta_{NYB} := \text{Hom}(\partial^{NYB}, \textrm{Id}_{A}).$$

\begin{definition}
Let $X$ be a cycle set and $A$ an abelian group. Then the following homology group and cohomology group
$$H_{n}^{NYB}(X;A)=H_{n}(C_{*}^{NYB}(X;A)),~ H^{n}_{NYB}(X;A)=H^{n}(C^{*}_{NYB}(X;A))$$
are called $n$th \emph{normalized set-theoretic Yang-Baxter homology group}\footnote{Its geometric realization and homotopical invariants obtained from it were discussed in \cite{WY}.} and the $n$th \emph{normalized set-theoretic Yang-Baxter cohomology group of $X$ with coefficient group $A.$}
\end{definition}

In a way similar to the idea shown in \cite{LN,NP}, we prove that the set-theoretic Yang-Baxter homology groups of certain solutions can be split into the normalized and degenerated parts.

For each $n,$ we consider the homomorphism $\kappa_{n}:C_{n}^{YB}(X) \rightarrow C_{n}^{YB}(X)$ defined by
$$\kappa_{n}(\textbf{x})=(x_{1}-x_{2}x_{2}) \otimes (x_{2}-x_{3}x_{3}) \otimes \cdots \otimes (x_{n-1}-x_{n}x_{n}) \otimes x_{n} $$
and extending linearly to all elements of $C_{n}^{YB}(X).$ It is easy to see that $\textrm{Id}_{C_{n}^{YB}(X)}-\kappa_{n}:C_{n}^{YB}(X) \rightarrow C_{n}^{D}(X)$ is a section for the short exact sequence of chain complexes $0 \rightarrow C_{*}^{D}(X) \rightarrow C_{*}^{YB}(X) \rightarrow C_{*}^{NYB}(X) \rightarrow 0$ (i.e., it is a split short exact sequence) and the short exact sequence stretches out to a long exact sequence of homology groups

$$\cdots \rightarrow H_{n+1}^{NYB}(X) \rightarrow H_{n}^{D}(X) \rightarrow H_{n}^{YB}(X) \rightarrow H_{n}^{NYB}(X) \rightarrow H_{n-1}^{D}(X) \rightarrow \cdots.$$

\begin{lemma}\label{lemma: splitting}
For a cyclic rack $X,$ the sequence $\kappa_{*}$ of the homomorphisms $\kappa_{n}:C_{n}^{YB}(X) \rightarrow C_{n}^{YB}(X)$ is a chain map.
\end{lemma}

\begin{proof}
Let $X$ be a cyclic rack. Since $X$ is a cycle set, we can obtain the corresponding solution of the Yang-Baxter equation $R:X\times X\to X\times X$ defined by $R(x,y)=(y+1,x-1)$ by Theorem \ref{thm: Rump}. We prove that $\partial_{n}^{YB} \circ \kappa_{n} = \kappa_{n-1} \circ \partial_{n}^{YB}$ for $n \geq 2$ inductively.\\
For $n=2,$ $\partial_{2}^{YB} \circ \kappa_{2} ((x_{1}, x_{2}))=\partial_{2}^{YB}((x_{1}, x_{2})-(x_{2}x_{2}, x_{2}))=\partial_{2}^{YB}((x_{1}, x_{2}))=\kappa_{1} \circ \partial_{2}^{YB} ((x_{1}, x_{2})).$ We denote by $\partial_{n}^{l}=\sum\limits_{j=1}^{n}(-1)^{j+1}d_{j,n}^{l},$ $\partial_{n}^{r}=\sum\limits_{j=1}^{n}(-1)^{j+1}d_{j,n}^{r},$ and $\textbf{x}=(x_{1}, \ldots, x_{n}).$ Note that $\partial_{n}^{YB}=\partial_{n}^{l}-\partial_{n}^{r}.$ Suppose that the statement is true for some $n \geq 2.$ More precisely, we assume that $\partial_{n}^{l} \circ \kappa_{n} = \kappa_{n-1} \circ \partial_{n}^{l}$ and $\partial_{n}^{r} \circ \kappa_{n} = \kappa_{n-1} \circ \partial_{n}^{r}.$ Then we have
\begin{align*}
\partial_{n+1}^{l} \circ \kappa_{n+1}(y \otimes \textbf{x})
&=\partial_{n+1}^{l}(y \otimes \kappa_{n}(\textbf{x}) - (x_{1}x_{1}) \otimes \kappa_{n}(\textbf{x}))\\
&=\kappa_{n}(\textbf{x}) - (y-1) \otimes \partial_{n}^{l}(\kappa_{n}(\textbf{x})) - \kappa_{n}(\textbf{x}) + (x_{1}x_{1}-1) \otimes \partial_{n}^{l}(\kappa_{n}(\textbf{x}))\\
&=\kappa_{n}(\textbf{x}) - \kappa_{n}(\textbf{x}) - ((y-1)-x_{1}) \otimes \partial_{n}^{l}(\kappa_{n}(\textbf{x}))\\
&=\kappa_{n}(\textbf{x}) - (x_{1}-x_{2}x_{2}) \otimes \kappa_{n-1}(x_{2}, \ldots, x_{n}) - ((y-1)-(x_{1}-1)^{2}) \otimes \kappa_{n-1}(\partial_{n}^{l}(\textbf{x}))\\
&=\kappa_{n}(\textbf{x}) + ((y-1)-(x_{1}-1)^{2}-(y-1)+x_{2}x_{2}) \otimes \kappa_{n-1}(x_{2}, \ldots, x_{n}) \\
&\hspace{5mm} - ((y-1)-(x_{1}-1)^{2}) \otimes \kappa_{n-1}(\partial_{n}^{l}(\textbf{x}))\\
&=\kappa_{n}(\textbf{x}) - \kappa_{n}((y-1) \otimes \partial_{n}^{l}(\textbf{x}))\\
&=\kappa_{n} \circ \partial_{n+1}^{l}(y \otimes \textbf{x})
\end{align*}
and
\begin{align*}
\partial_{n+1}^{r} \circ \kappa_{n+1}(y \otimes \textbf{x})
&=\partial_{n+1}^{r}( (y - x_{1}x_{1}) \otimes \kappa_{n}(\textbf{x}))\\
&=d_{1,n+1}^{r}((y-x_{1}x_{1}) \otimes \kappa_{n}(\textbf{x})) - (y - x_{1}x_{1}) \otimes (\partial_{n}^{r} \circ \kappa_{n}(\textbf{x}))\\
&=- (y - x_{1}x_{1}) \otimes (\kappa_{n-1} \circ \partial_{n}^{r}(\textbf{x}))\\
&=- (y - x_{1}x_{1}) \otimes (\kappa_{n-1} \circ \partial_{n}^{r}(\textbf{x})) + \kappa_{n}(x_{1}+1, x_{2}+1, \ldots , x_{n}+1)\\
&\hspace{5mm} - ((x_{1}+1) -y +y - (x_{2}+1)(x_{2}+1)) \otimes \kappa_{n-1}(x_{2}+1, \ldots , x_{n}+1)\\
&=\kappa_{n}( d_{1,n+1}^{r}(y \otimes \textbf{x}) - y \otimes \partial_{n}^{r}(\textbf{x}) )\\
&=\kappa_{n} \circ \partial_{n+1}^{r}(y \otimes \textbf{x}).
\end{align*}
Therefore, $\partial_{n+1}^{YB} \circ \kappa_{n+1} = \kappa_{n} \circ \partial_{n+1}^{YB}$ as desired.
\end{proof}

\begin{theorem}
Let $X$ be a cyclic rack. Then the set-theoretic Yang-Baxter homology of $X$ splits into the normalized and degenerate parts, that is
$$H_{*}^{YB}(X)=H_{*}^{NYB}(X) \oplus H_{*}^{D}(X).$$
\end{theorem}

\begin{proof}
By Lemma \ref{lemma: splitting}, $\textrm{Id}_{C_{n}^{YB}(X)}-\kappa_{n}:C_{n}^{YB}(X) \rightarrow C_{n}^{D}(X)$ is a chain map. Then the composition of the map $H_{n}^{D}(X) \rightarrow H_{n}^{YB}(X)$ and the homomorphism induced from $\textrm{Id}_{C_{n}^{YB}(X)}-\kappa_{n}$ forms the identity map on $H_{n}^{D}(X).$ Therefore, by the splitting lemma, we have $H_{*}^{YB}(X)=H_{*}^{NYB}(X) \oplus H_{*}^{D}(X).$
\end{proof}

\begin{remark*}
Decompositions in homology of a semi-strong skew cubical structure were discussed in \cite{LV}. Calculating in a cyclic rack, it is not difficult to show that conditions (2.5) and (2.6) of \cite{LV} fail.
\end{remark*}

\begin{example}
Table \ref{Tb:H} summarizes some computational results on (normalized) set-theoretic Yang-Baxter homology groups. Here, $C_{4}$ is the cyclic rack of order $4$, $R_{4}$ is the dihedral quandle of order $4$, $X_{4}$ is the affine cycle set $\mathrm{Aff}(\mathbb F_2^2,\binom{1\ 0}{1\ 1},\binom{0\ 1}{1\ 0},\binom{0}{0})$ (depicted in Table \ref{table 1}), and $X_{16}$ is the affine cycle set $\mathrm{Aff}(\mathbb F_4^2,\binom{0\ u}{u^2\ 0},\binom{0\ 1}{1\ 0},\binom{0}{0})$, where $u$ is a primitive element of $\mathbb F_4$ (depicted in Table \ref{table 2}).
\end{example}

\begin{table}[h]
  \centering
  \caption{An affine cycle set of order $16$}\label{table 2}
  \begin{tabular}{c|cccccccccccccccc}
 $\ast$ & 1 & 2 & 3 & 4 & 5 & 6 & 7 & 8 & 9 & 10 & 11 & 12 & 13 & 14 & 15 & 16\\
 \hline
 1 & 1 & 5 & 9 & 13 & 2 & 6 & 10 & 14 & 3 & 7 & 11 & 15 & 4 & 8 & 12 & 16 \\
 2 & 9 & 13 & 1 & 5 & 10 & 14 & 2 & 6 & 11 & 15 & 3 & 7 & 12 & 16 & 4 & 8 \\
 3 & 13 & 9 & 5 & 1 & 14 & 10 & 6 & 2 & 15 & 11 & 7 & 3 & 16 & 12 & 8 & 4 \\
 4 & 5 & 1 & 13 & 9 & 6 & 2 & 14 & 10 & 7 & 3 & 15 & 11 & 8 & 4 & 16 & 12 \\
 5 & 4 & 8 & 12 & 16 & 3 & 7 & 11 & 15 & 2 & 6 & 10 & 14 & 1 & 5 & 9 & 13 \\
 6 & 12 & 16 & 4 & 8 & 11 & 15 & 3 & 7 & 10 & 14 & 2 & 6 & 9 & 13 & 1 & 5 \\
 7 & 16 & 12 & 8 & 4 & 15 & 11 & 7 & 3 & 14 & 10 & 6 & 2 & 13 & 9 & 5 & 1 \\
 8 & 8 & 4 & 16 & 12 & 7 & 3 & 15 & 11 & 6 & 2 & 14 & 10 & 5 & 1 & 13 & 9 \\
 9 & 2 & 6 & 10 & 14 & 1 & 5 & 9 & 13 & 4 & 8 & 12 & 16 & 3 & 7 & 11 & 15 \\
 10 & 10 & 14 & 2 & 6 & 9 & 13 & 1 & 5 & 12 & 16 & 4 & 8 & 11 & 15 & 3 & 7 \\
 11 & 14 & 10 & 6 & 2 & 13 & 9 & 5 & 1 & 16 & 12 & 8 & 4 & 15 & 11 & 7 & 3 \\
 12 & 6 & 2 & 14 & 10 & 5 & 1 & 13 & 9 & 8 & 4 & 16 & 12 & 7 & 3 & 15 & 11 \\
 13 & 3 & 7 & 11 & 15 & 4 & 8 & 12 & 16 & 1 & 5 & 9 & 13 & 2 & 6 & 10 & 14 \\
 14 & 11 & 15 & 3 & 7 & 12 & 16 & 4 & 8 & 9 & 13 & 1 & 5 & 10 & 14 & 2 & 6 \\
 15 & 15 & 11 & 7 & 3 & 16 & 12 & 8 & 4 & 13 & 9 & 5 & 1 & 14 & 10 & 6 & 2 \\
 16 & 7 & 3 & 15 & 11 & 8 & 4 & 16 & 12 & 5 & 1 & 13 & 9 & 6 & 2 & 14 & 10

\end{tabular}
\end{table}

\def\Z{\mathbb Z}

\begin{table}
\centering
\caption{(Normalized) set-theoretic Yang-Baxter homology groups}\label{Tb:H}
\ra{1.3}
\begin{tabular}{@{}llllll@{}}
    \toprule
    $n$&$1$&$2$&$3$&$4$&$5$\\
    \midrule
    $H_n^{YB}(C_3)$       &$\Z\oplus\Z_3$   &$\Z^3$                             &$\Z^{9}\oplus\Z_3$                &$\Z^{27}$                          &$\Z^{81}\oplus\Z_3$\\
    $H_n^D(C_3)$          &$0$              &$\Z$                               &$\Z^5$                             &$\Z^{19}$                          &$\Z^{65}$\\
    $H_n^{NYB}(C_3)$      &$\Z\oplus\Z_3$   &$\Z^2$                             &$\Z^4\oplus\Z_3$                   &$\Z^{8}$                          &$\Z^{16}\oplus\Z_3$\\
    &&&&&\\
    $H_n^{YB}(C_4)$       &$\Z\oplus\Z_4$   &$\Z^4$                             &$\Z^{16}\oplus\Z_4$                &$\Z^{64}$                          &$\Z^{256}\oplus\Z_4$\\
    $H_n^D(C_4)$          &$0$              &$\Z$                               &$\Z^7$                             &$\Z^{37}$                          &$\Z^{175}$\\
    $H_n^{NYB}(C_4)$      &$\Z\oplus\Z_4$   &$\Z^3$                             &$\Z^9\oplus\Z_4$                   &$\Z^{27}$                          &$\Z^{81}\oplus\Z_4$\\
    &&&&&\\
    $H_n^{YB}(C_5)$       &$\Z\oplus\Z_5$   &$\Z^5$                             &$\Z^{25}\oplus\Z_5$                &$\Z^{125}$                          &$\Z^{625}\oplus\Z_5$\\
    $H_n^D(C_5)$          &$0$              &$\Z$                               &$\Z^{9}$                             &$\Z^{61}$                          &$\Z^{369}$\\
    $H_n^{NYB}(C_5)$      &$\Z\oplus\Z_5$   &$\Z^4$                             &$\Z^{16}\oplus\Z_5$                   &$\Z^{64}$                          &$\Z^{256}\oplus\Z_5$\\
    &&&&&\\
    $H_n^{YB}(R_4)$       &$\Z^2\oplus\Z_2$ &$\Z^6\oplus\Z_2$                   &$\Z^{20}\oplus\Z_2^3\oplus\Z_4$    &$\Z^{72}\oplus\Z_2^7\oplus\Z_4$    &$\Z^{272}\oplus\Z_{2}^{17} \oplus\Z_{4}^{2}$\\
    $H_n^D(R_4)$          &$0$              &$\Z^2$                             &$\Z^{10}\oplus\Z_2^2$              &$\Z^{44}\oplus\Z_2^6$              &$\Z^{190}\oplus\Z_{2}^{16}$\\
    $H_n^{NYB}(R_4)$      &$\Z^2\oplus\Z_2$ &$\Z^4\oplus\Z_2$                   &$\Z^{10}\oplus\Z_2\oplus\Z_4$      &$\Z^{28}\oplus\Z_2\oplus\Z_4$      &$\Z^{82}\oplus\Z_{2} \oplus\Z_{4}^{2}$\\
    &&&&&\\
    $H_n^{YB}(X_4)$       &$\Z\oplus\Z_2$   &$\Z^3\oplus\Z_2$                   &$\Z^{10}\oplus\Z_2^2$              &$\Z^{36}\oplus\Z_2^3\oplus\Z_4$    &$\Z^{136}\oplus\Z_2^5$\\
    $H_n^D(X_4)$          &$0$              &$\Z$                               &$\Z^5$                             &$\Z^{22}\oplus\Z_2$                &$\Z^{95}\oplus\Z_2^2$\\
    $H_n^{NYB}(X_4)$      &$\Z\oplus\Z_2$   &$\Z^2\oplus\Z_2$                   &$\Z^5\oplus\Z_2^2$                 &$\Z^{14}\oplus\Z_2^2\oplus\Z_4$    &$\Z^{41}\oplus\Z_2^3$\\
    &&&&&\\
    $H_n^{YB}(X_{16})$    &$\Z\oplus\Z_2^3$ &$\Z^{10}\oplus\Z_2^6\oplus\Z_4^3$  & $\Z^{136}\oplus\Z_{2}^{30}\oplus\Z_{4}$   &                                   &\\
    $H_n^{D}(X_{16})$     &$0$              &$\Z$                               & $\Z^{19}$                           &                                   &\\
    $H_n^{NYB}(X_{16})$   &$\Z\oplus\Z_2^3$ &$\Z^9\oplus\Z_2^6\oplus\Z_4^3$     & $\Z^{117}\oplus\Z_{2}^{30}\oplus\Z_{4}$   &                                   &\\
    \bottomrule
\end{tabular}
\end{table}
Note that for the dihedral quandle of order $4,$ its rack/quandle homology groups do not contain $\mathbb{Z}_{4}$-torsion, while its set-theoretic/normalized set-theoretic Yang-Baxter homology groups do.

The following are some computational results on the (Normalized) set-theoretic Yang-Baxter homology groups of cyclic racks.

\begin{theorem}\label{thm:C2}
For the cyclic rack $C_{2}$ of order $2,$
$$
H_{n}^{NYB}(C_{2}) =
\left\{
  \begin{array}{ll}
    \mathbb{Z} \oplus \mathbb{Z}_{2}, & \hbox{$n$ is odd;} \\
    \mathbb{Z}, & \hbox{$n$ is even.}
  \end{array}
\right.
$$
\end{theorem}

\begin{proof}
Note that for every dimension $n,$ $C_{n}^{NYB}(C_{2})$ has two generators $(0, \ldots, 0)$ and $(1, \ldots, 1).$
\begin{align*}
\partial_{n}^{NYB}(x_{1}, \ldots, x_{n})
&=(x_{2}, \ldots , x_{n}) - (x_{2}+1, \ldots , x_{n}+1)\\
&+ (-1)^{n-1}(x_{1}-1, \ldots , x_{n-1}-1) + (-1)^{n}(x_{1}, \ldots , x_{n-1})\\
&=(x_{1}, \ldots , x_{n-1})  + (-1)^{n}(x_{1}, \ldots , x_{n-1})\\
&+ (-1)^{n-1}(x_{1}+1, \ldots , x_{n-1}+1) - (x_{1}+1, \ldots , x_{n-1}+1)\\
&=\left\{
    \begin{array}{ll}
      0 & \hbox{if $n$ is odd;} \\
      2(x_{1}, \ldots , x_{n-1}) - 2(x_{1}+1, \ldots , x_{n-1}+1) & \hbox{if $n$ is even.}
    \end{array}
  \right.
\end{align*}
Therefore,\\
$H_{n}^{NYB}(C_{2}) \cong
\left\{
\begin{array}{ll}
\langle (0, \ldots, 0), (0, \ldots, 0) - (1, \ldots, 1) ~|~ 2\{(0, \ldots, 0) - (1, \ldots, 1)\} \rangle & \hbox{if $n$ is odd;} \\
\langle(0, \ldots, 0) + (1, \ldots, 1) ~|~ - \rangle & \hbox{if $n$ is even}
\end{array}
\right.
$\\
as desired.
\end{proof}

\begin{theorem}\label{thm:Cm}
Let $C_{m}$ be the cyclic rack of order $m.$ Then $H_{1}^{NYB}(C_{m}) = \mathbb{Z} \oplus \mathbb{Z}_{m}.$
\end{theorem}

\begin{proof}
Since $\partial_{2}^{NYB}(x_{1}, x_{2}) = (x_{2}) - (x_{2}+1) - (x_{1}-1) + (x_{1}),$ one can check that
$$\partial_{2}^{NYB}(x, x+k) = \sum\limits_{i=0}^{k}\partial_{2}^{NYB}(x+i, x+i)$$
for each $k=0, \ldots , m-2.$
Thus, we can perform column operations on $\partial_{2}^{NYB}$ in order to obtain
$\left(
                                \begin{array}{ccc}
                                  C & \vdots & O \\
                                \end{array}
                              \right)
,$
where $O$ is a zero matrix and \\
$C = \left(
             \begin{array}{cccccc}
               2 & -1 &   &  & 0 & -1 \\
               -1 & 2 &   &  & \vdots & 0 \\
               0 & -1 &  \ddots &  & 0 & \vdots \\
               \vdots & 0 &  & \ddots & -1 & 0 \\
               0 & \vdots &  &  & 2 & -1 \\
               -1 & 0 &   &  & -1 & 2 \\
             \end{array}
           \right)$
is a circulant matrix.\\
Note that the Smith normal form of $C$ is $\left(
                                             \begin{array}{ccc}
                                               I_{m-2} & O & O \\
                                               O & m & 0 \\
                                               O & 0 & 0 \\
                                             \end{array}
                                           \right)
,$ where $I_{m-2}$ is the identity matrix of size $m-2.$ Therefore, $H_{1}^{NYB}(C_{m}) = \mathbb{Z} \oplus \mathbb{Z}_{m}.$
\end{proof}

Theorem \ref{thm:C2} and Theorem \ref{thm:Cm} together with our computational data in Table \ref{Tb:H} support the following conjecture.

\begin{conjecture}
For the cyclic rack $C_{m}$ of order $m,$\\
$$H_{n}^{NYB}(C_{m}) = \left\{
                        \begin{array}{ll}
                          \mathbb{Z}^{((m-1)^{n-1})} \oplus \mathbb{Z}_{m} & \hbox{if $n$ is odd;} \\
                          \mathbb{Z}^{((m-1)^{n-1})} & \hbox{if $n$ is even.}
                        \end{array}
                      \right.
$$
\end{conjecture}

\section{Cocycle link invariants obtained from the normalized set-theoretic Yang-Baxter homology}\label{Section 3}

Biquandles\cite{FRS, KR} can be used to construct cocycle invariants of knots and links\cite{CES}. For a given cycle set $X,$ we consider the corresponding involutive right non-degenerate solution $R(x,y) = (y(x / y),x / y).$ If it is also left non-degenerate, then the algebraic structure $(X, *_{1}, *_{2})$ equipped with two binary operations $*_{1}, *_{2}$ on $X$ defined by $x *_{1} y = y(x / y)$ and $x *_{2} y = x / y$ becomes a biquandle. We can, therefore, use finite cycle sets to define cocycle link invariants as Carter, Elhamdadi, and Saito\cite{CES} did, and moreover higher dimensional cocycles are expected to be used to construct invariants of higher dimensional knots. See \cite{PR} for further details. For $3$-dimensional links, a detailed construction is as follows.

Let $D_{L}$ be an oriented link diagram of a given oriented link $L,$ and let $\mathcal{R}$ be the set of all semiarcs of $D_{L}.$ Given a finite cycle set $X,$ the map $\mathcal{C}_{D_{L}}:\mathcal{R} \rightarrow X$ satisfying the relation depicted in Figure \ref{cyclesetcoloring} at each crossing of $D_{L}$ is said to be a \emph{cycle set coloring} of a link diagram $D_{L}$ by $X.$

\begin{figure}[h]
\centerline{{\psfig{figure=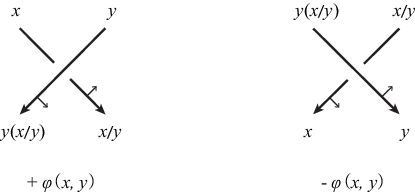,height=4cm}}}
\caption{Cycle set coloring relations at each crossing}
\label{cyclesetcoloring}
\end{figure}

Suppose that $X$ is a finite cycle set. Let $\varphi$ be a $2$-cocycle of $C^{2}_{NYB}(X;A)$ with coefficients in an abelian group $A.$ For a given cycle set coloring $\mathcal{C}_{D_{L}},$ a Boltzmann weight, denoted by $B_{\varphi}(\mathcal{C}_{D_{L}}),$ associated with $\varphi$ is defined as follows. We assign the weight $\varepsilon(\tau)\varphi(x,y)$ for each crossing $\tau,$ where $\varepsilon(\tau)=1$ or $-1$ if $\tau$ is a positive crossing or a negative crossing, respectively and $x$ and $y$ denote the colors of the semiarcs as depicted in Figure \ref{cyclesetcoloring}. The Boltzmann weight $B_{\varphi}(\mathcal{C}_{D_{L}})$ is then defined by $B_{\varphi}(\mathcal{C}_{D_{L}})=\sum\limits_{\tau} \varepsilon(\tau)\varphi(x,y).$\\
We now consider the value $\Phi(L)=\sum\limits_{\mathcal{C}_{D_{L}}}B_{\varphi}(\mathcal{C}_{D_{L}}) \in \mathbb{Z}[A],$ where $\mathbb{Z}[A]$ is the group ring of $A$ over $\mathbb{Z}.$

\begin{proposition}\label{link invariant}
$\Phi(L)$ is an oriented link invariant.
\end{proposition}

Before proving the above theorem, let us consider cycle set colorings for oriented Reidemeister moves of type I. We start with the following lemma.

\begin{lemma}
Let $X$ be a finite cycle set. Given $x\in X$, there is a unique $y\in X$ such that $x(y/x)=y$.
\end{lemma}

\begin{proof}
It is clear that $y=xx$ solves the equation. We next prove the uniqueness. With $z=y/x$ we have $y=zx$, and the equation becomes $xz=zx.$ By Proposition \ref{Prop:Rump} the $\Delta$-map $\Delta(a,b)=(ab,ba)$ is bijective. Note that it maps the diagonal onto itself. Assuming we have a solution, the $\Delta$-map assigns $(x,z)$ to $(xz,zx)$ which is a diagonal element. Hence $x=z$ and $y=xx.$
\end{proof}

According to the cycle set coloring convention and the above lemma, we can color the type I moves as shown in Figure \ref{R1move}.

\begin{figure}[h]
\centerline{{\psfig{figure=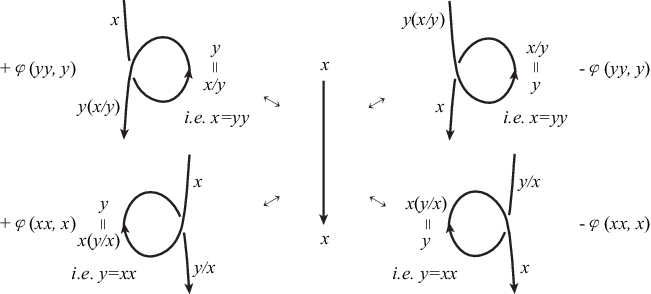,height=6.3cm}}}
\caption{Weights for Reidemeister Type I moves}
\label{R1move}
\end{figure}

\begin{proof}[Proof of Proposition \ref{link invariant}]
Let $\varphi$ be a $2$-cocycle in $C^{2}_{NYB}(X;A).$ Since there is a one-to-one correspondence between cycle set colorings of the diagrams before and after each Reidemeister move, it is sufficient to show that the Boltzmann weight remains unchanged under each Reidemeister move.

The weights assigned to the crossings of Reidemeister moves of type I have a form of $\pm\varphi(xx,x),$ see Figure \ref{R1move}, which is $0$ in $A$ because $(xx,x), (yy,y) \in C_{2}^{D}(X).$ Therefore, the Boltzmann weight is invariant under Reidemeister moves of Type I.

For a Reidemeister type II move, the weights assigned to the crossings are the same but have opposite signs, see Figure \ref{R2move} for example. The Boltzmann weight therefore is unchanged by the type II move.

\begin{figure}[h]
\centerline{{\psfig{figure=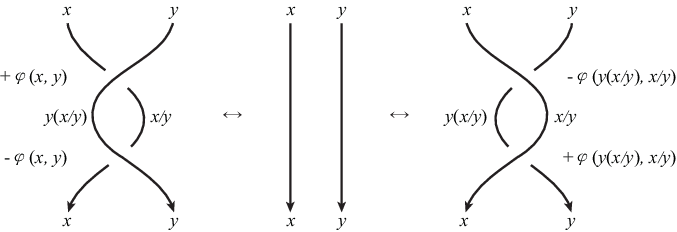,height=4.5cm}}}
\caption{Weights for Reidemeister Type II moves}
\label{R2move}
\end{figure}

Since $\varphi$ is a $2$-cocycle, for any $(x,y,z) \in C_{3}^{NYB}(X)$ we have
\begin{align*}
0
&=\delta_{2}(\varphi)(x,y,z)=\varphi \circ \partial_{3}^{NYB}(x,y,z)\\
&=\varphi\{(y,z)- (y(x/y), z((x/y)/z)) - (x/y, z) + (x, z(y/z)) + (x/(z(y/z)), y/z) -(x,y)\}
\end{align*}
which is the difference between the weights for the diagrams before and after a Reidemeister move of type III, see Figure \ref{R3move} for instance. In a similar way, one can show that the sum of weights remain unchanged under other possible type III moves. Therefore, the Boltzmann weight is invariant under type III moves.

\begin{figure}[h]
\centerline{{\psfig{figure=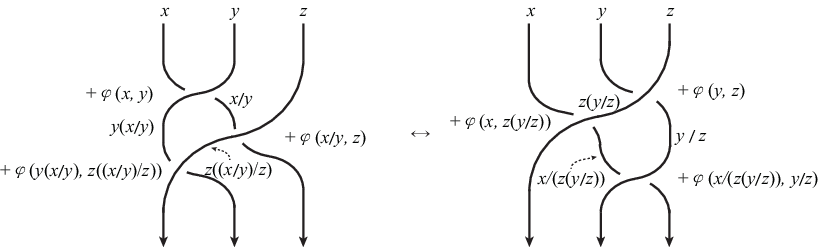,height=5cm}}}
\caption{Weights for Reidemeister Type III moves}
\label{R3move}
\end{figure}

\end{proof}

We investigate some non-trivial cocycles of the cyclic racks and use them to classify torus links.

\begin{theorem}
Let $C_{m}$ be the cyclic rack of order $m.$ For $n \geq 2,$ consider the maps $\theta_{n} \in C_{NYB}^{n}(C_{m}; \mathbb{Z}_{2})$ and $\xi_{n} \in C_{NYB}^{n}(C_{m}; \mathbb{Z}_{m})$ defined by
$$\theta_{n}(x_{1}, \ldots, x_{n})=\prod_{i=1}^{n}x_{i} ~(\hbox{\emph{mod}~} 2),$$
$$\xi_{n}(x_{1}, \ldots, x_{n})=\prod_{i=1}^{n-1}(x_{i}-x_{i+1}-1)~(\hbox{\emph{mod}~} m)$$
and extending linearly to all elements of $C^{NYB}_{n}(C_{m}).$ Then $\theta_{n}$ is an $n$-cocycle when $m$ is even and $\xi_{n}$ is an $n$-cocycle for any $m.$
\end{theorem}

\begin{proof}
(i) Note that $xx=x+1$ in $C_{m}.$ For every $(x_{1}, \ldots, x_{n}) \in C_{n}^{D}(C_{m}),$ $\prod\limits_{i=1}^{n}x_{i}$ is even, i.e., $\theta_{n}(x_{1}, \ldots, x_{n})=0~(\hbox{mod~} 2).$ We now only need to show that $\theta_{n} \circ \partial_{n+1}^{YB}=0.$\\
Let $(x_{1}, \ldots, x_{n+1}) \in C_{n+1}^{YB}(C_{m}).$ Since $m$ is even, $x$ and $x \pm 1$ have different parity in $C_{m}.$ Then we have
\begin{align*}
\theta_{n} \circ \partial_{n+1}^{YB}(x_{1}, \ldots, x_{n+1})
&=\theta_{n}\{(x_{2}, \ldots, x_{n+1}) - (x_{2}+1, \ldots, x_{n+1}+1)\\
&-(x_{1}-1, x_{3}, \ldots, x_{n+1}) + (x_{1}, x_{3}+1, \ldots, x_{n+1}+1)\\
&+ \cdots +(-1)^{n+1}(x_{1}-1, \ldots, x_{n}-1) - (-1)^{n+1}(x_{1}, \ldots, x_{n}) \}\\
&\equiv x_{2} \cdots x_{n+1} + (x_{2}+1) \cdots (x_{n+1}+1)\\
&+(x_{1}+1)x_{3} \cdots x_{n+1} + x_{1}(x_{3}+1) \cdots (x_{n+1}+1)\\
&+ \cdots + (x_{1}+1) \cdots (x_{n}+1) + x_{1} \cdots x_{n}\\
&= 2\left\{ \prod\limits_{i=1}^{n+1}(x_{i}+1) - \prod\limits_{i=1}^{n+1}(x_{i}) \right\}\\
&\equiv 0 ~(\hbox{mod~} 2).
\end{align*}
\\
(ii) It is easy to see that $\xi_{n}(x_{1}, \ldots, x_{n})=0~(\hbox{mod~} m)$ for every $(x_{1}, \ldots, x_{n}) \in C_{n}^{D}(C_{m}).$
Next, we show that $\xi_{n} \circ \partial_{n+1}^{YB} = 0.$
\begin{align*}
&\xi_{n} \circ (d_{i,n+1}^{l} - d_{i,n+1}^{r})(x_{1}, \ldots, x_{n+1})\\
&=\xi_{n}(x_{1}-1, \ldots, x_{i-1}-1, x_{i+1}, \ldots, x_{n+1})\\
&- \xi_{n}(x_{1}, \ldots, x_{i-1}, x_{i+1}+1, \ldots, x_{n+1}+1)\\
&=(x_{1}-x_{2}-1) \cdots (x_{i-2}-x_{i-1}-1)(x_{i-1}-x_{i+1}-2)(x_{i+1}-x_{i+2}-1) \cdots (x_{n}-x_{n+1}-1)\\
&-(x_{1}-x_{2}-1) \cdots (x_{i-2}-x_{i-1}-1)(x_{i-1}-x_{i+1}-2)(x_{i+1}-x_{i+2}-1) \cdots (x_{n}-x_{n+1}-1)\\
&=0.
\end{align*}
Therefore, $\xi_{n} \circ \partial_{n+1}^{YB} = \xi_{n} \circ \left( \sum\limits_{i=1}^{n+1}(-1)^{i+1}(d_{i,n+1}^{l} - d_{i,n+1}^{r}) \right) =0$ as desired.
\end{proof}

One can use the cocycles obtained from above to classify knots and links.

\begin{example}
Let $T(p,q)$ be the $(p,q)$-torus link and let $C_{m}$ be the cyclic rack of order $m.$ Consider the $2$-cocycle $\varphi \in \textrm{Hom}(C_{2}^{NYB}(C_{m}), \mathbb{Z}_{2})$ defined by
$$\varphi(x,y) \equiv xy ~ (\textrm{mod~} 2).$$
For $n \in \mathbb{N},$ we compute \footnote{Note that the value of the trivial link with $s$ components is $m^{s}(0)$ in $\mathbb{Z}[\mathbb{Z}_{2}].$}\\
\begin{enumerate}
  \item $\Phi(T(2, q)) =
\left\{
  \begin{array}{ll}
    m(0), & \hbox{$q=2n-1$;} \\
    m^{2}(0), & \hbox{$q=2(2n)$;} \\
    \frac{m^{2}}{2}(0) + \frac{m^{2}}{2}(1), & \hbox{$q=2(2n-1)$.}
  \end{array}
\right.
$
\vspace{5mm}
  \item $\Phi(T(3, q)) =
\left\{
  \begin{array}{ll}
    m(0), & \hbox{$q=3n-2$ or $3n-1$;} \\
    m^{3}(0), & \hbox{$q=3n$.}
  \end{array}
\right.
$
\vspace{5mm}
  \item $\Phi(T(4, q)) =
\left\{
  \begin{array}{ll}
    m(0), & \hbox{$q=4n-3$ or $4n-1$;} \\
    m^{2}(0), & \hbox{$q=4n-2$;} \\
    m^{4}(0), & \hbox{$q=4(2n)$;} \\
    \frac{m^{4}}{2}(0) + \frac{m^{4}}{2}(1), & \hbox{$q=4(2n-1)$.}
  \end{array}
\right.
$
\vspace{5mm}
  \item $\Phi(T(5, q)) =
\left\{
  \begin{array}{ll}
    m(0), & \hbox{$q=5n-4$, $5n-3$, $5n-2$, or $5n-1$;} \\
    m^{5}(0), & \hbox{$q=5n$.}
  \end{array}
\right.
$
\vspace{5mm}
  \item $\Phi(T(6, q)) =
\left\{
  \begin{array}{ll}
    m(0), & \hbox{$q=6n-5$ or $6n-1$;} \\
    m^{3}(0), & \hbox{$q=6n-3$;} \\
    m^{2}(0), & \hbox{$q=6(2n-1)+2$ or $q=6(2n-2)+4$;} \\
    \frac{m^{2}}{2}(0) + \frac{m^{2}}{2}(1), & \hbox{$q=6(2n-2)+2$ or $q=6(2n-1)+4$;} \\
    m^{6}(0), & \hbox{$q=6(2n)$;} \\
    \frac{m^{6}}{2}(0) + \frac{m^{6}}{2}(1), & \hbox{$q=6(2n-1)$.}
  \end{array}
\right.
$
\end{enumerate}

\end{example}

\section*{Acknowledgements}
J. H. Przytycki was partially supported by the Simons Collaboration Grant-316446 and CCAS Dean's Research Chair award. The work of Petr Vojt\v{e}chovsk\'y was supported by the PROF grant of the University of Denver. The work of Seung Yeop Yang was supported by the National Research Foundation of Korea(NRF) grant funded by the Korea government(MSIT) (No. 2019R1C1C1007402).

\end{document}